\documentclass[]{interact}

\usepackage{epstopdf}
\usepackage[caption=false]{subfig}
\usepackage{xcolor}

\usepackage[numbers,sort&compress]{natbib}
\bibpunct[, ]{[}{]}{,}{n}{,}{,}

\theoremstyle{plain}
\newtheorem{theorem}{Theorem}

\newtheorem{corollary}[theorem]{Corollary}
\newtheorem{proposition}[theorem]{Proposition}

\theoremstyle{definition}

\theoremstyle{remark}
\newtheorem{remark}{Remark}

\newcommand{\quot}[1]{`#1'}
\newcommand{\abs}[1]{\left\vert#1\right\vert}
\def\R{\mathbb{R}}

\begin{document}


\title{Asymptotic values of solutions to a periodic linear difference equation modeling discrimination training}

\author{
\name{Natham Aguirre\thanks{CONTACT Natham Aguirre. Email: nmaguirre@uc.cl}\thanks{https://orcid.org/0000-0001-7952-8286}}
\affil{Pontificia Universidad Católica de Chile, Avda. Vicuña Mackenna 4860, Macul, Chile}
}

\maketitle

\begin{abstract}
This work is concerned with the study of $w(mT)$ as $m$ goes to infinity, where $w(t)$ evolves according to $w(t)-w(t-1)=F(t)-A(t)w(t-1)$, and where $T$ is the period of the vector $F(t)$ and the matrix $A(t)$.
Motivated by applications to associative learning, particularly to discrimination training, extra conditions are imposed on $F(t)$ and $A(t)$, one of them relating $A(t)$ to a symmetric non-negative definite matrix $K$ relevant to mathematical models of associative learning.
Structural relationships between the matrices imply an identity satisfied by the Floquet multipliers driving the dynamics of $w(mT)$ from which follows that the unstable subspace is $\ker K$.
Then, the limit of $w(mT)$ is explicitly identified when $K$ is invertible, while the limit of $Kw(mT)$ is established otherwise. Given that divergence of $w(mT)$ can happen when $K$ is singular, while $Kw(mT)$ is the psychologically relevant quantity, the result can be considered optimal.
\end{abstract}


\section{Introduction}

This works concerns the study of solutions $w(t)$ to a difference equation of the form
\begin{equation}\label{eq}
w(t)-w(t-1)=F(t)-A(t)w(t-1) \quad , \quad t\in \mathbb{N}
\end{equation}
where $F(t)$ is a forcing vector in $\mathbb{R}^{n}$, $A(t)$ is a $n\times n$ real coefficient matrix, and both $F(t)$ and $A(t)$ are assumed to be periodic.

Motivated by the application of this equation to the modeling of certain behavioral phenomena (which is described below), it is assumed that there exist times $0=T_0<T_1<\cdots<T_n$ such that $F(t)$ and $A(t)$ are periodic with period $T_n$, and that in the time interval $[T_{i-1}+1, T_i]$ (for $i=1,\cdots,n$)
\begin{enumerate}
\item[i)] $F(t)$ and $A(t)$ are constant,
\item[ii)] the entry in position $(i,i)$ of $A(t)$ is non-zero,
\item[iii)] rows different from row $i$ are zero in both $F(t)$ and $A(t)$.
\end{enumerate}
Moreover, I assume that
\begin{enumerate}
\item[iv)] the entries of $A(t)$ belong to the interval $[0,1)$,
\item[v)] there exists a diagonal matrix $D$ with positive entries such that the matrix $K$ given by
\[
K=D\sum_{i=1}^{n}A(T_i)
\]
is symmetric and non-negative definite.
\end{enumerate}
Note that the periodicity assumptions make \eqref{eq} a particular case of an inhomogeneous Floquet-type system. 

The main objective of this work is to study the evolution of $w$ at times $t=mT_n$ and to determine, when possible, the limit of $Kw(mT_n)$ as $m\rightarrow \infty$. Thus, this work may be considered related to those concerned with the study of asymptotic behavior of Floquet-type systems in the context of mathematical modeling \cite{SGMNP16,K22,ST11} or control theory \cite{DT98,LY97}. On the other hand, the particular case of equation \eqref{eq} with periodic coefficients and forcing has been studied in \cite{NNS08}. There, the authors derive a representation of solutions based on Floquet multipliers, which in turn leads to several general results regarding boundedness and periodicity of solutions. However, the particularity of the assumptions and objectives proposed make this work deviate from the method in \cite{NNS08}.   

The assumptions imposed on the forcing vector $F(t)$ and coefficient matrix $A(t)$, as well as the interest on $w(mT_n)$ and $Kw(mT_n)$, are motivated by the mathematical modeling of certain behavioral phenomena of interest in the field of associative learning. 

Broadly, associative learning is concerned with the study of how, as a result of experience, organisms may come to link (or associate) different events \cite{WM97}. In general, it is proposed that the concurrent presentation of stimuli creates an association (at the level of representations) between them, and that these associations may alter behavior \cite{RW72,W03}.
For example, Pavlov \cite{Pavlov1927} noted that if a certain sound (initially with no visible effect on the behavior of an organism) is repeatedly presented along with food, which naturally produces a salivary response, then the sound will come to produce a salivary response by itself. 
From a theoretical point of view, an association between sound and food is developed as a result of the repeated concurrent presentation of the two stimuli. Hence, afterwards, presentation of the sound activates an internal representation of food which in turn leads to a salivary response.

To explain changes in the association between a stimulus $x$ (such as a sound) and a fixed target stimulus (such as food) some theories have developed mathematical models where the strength of the association is quantified as an \quot{associative value} (e.g. \cite{RW72, P94, TL20, W03, G20}). These models are normally used in training situations where a set of different stimuli, such as tones or lights of certain frequencies, each called a \quot{conditioned stimulus}, are being trained along with a fixed biologically relevant stimulus (such as food) which is called the \quot{unconditioned stimulus}. 

As it turns out, some of the most prominent of these models can be characterized as a particular class of \quot{kernel machines} \cite{G18} (i.e., an infinite dimensional perceptron, \cite{JSW07}) which measure the associative value $V(x)$ of a conditioned stimulus $x$ as 
\[
V(x)=\sum_{y\in \mathbb{T}}k(y,x)w_y
\]
where $k$ is a model-specific positive definite kernel (i.e., a symmetric function such that for any finite choice of $x_i$'s the matrix $[k(x_i,x_j)]_{i,j}$ is non-negative definite), $\mathbb{T}$ is the set of all relevant conditioned stimuli being trained, and $w_y$ are weights \cite{A25,A24,G18,G15}. 

In turn, training is assumed to change associative values by changing the values of the weights. Concretely, upon any discrete training event involving $x$, the weight $w_x$ (and only this weight) is updated according to the \quot{error correction rule}
\[
\Delta w_x=\delta_x (\Lambda-V(x))
\]
where $\delta_x\in (0,1)$ is a learning rate parameter that could depend on $x$ and the type of training event, and $\Lambda$ is either a positive number $\lambda$ or $0$ according to if the unconditioned stimulus was or not present during the training event (if it was present, the training event is called a \quot{reinforcement}; if not, it is called a \quot{non-reinforcement}). 

To get a sense of how these models work, the first principle to have in mind is that an associative value close to $\lambda$ should correlate with the conditioned stimulus having strong effect on behavior, while an associative value close to $0$ should correlate with the conditioned stimulus having little or no effect on behavior.
Then, note that in the case of only one conditioned stimulus $x$ starting at $V(x)=0$ (and with $0<\delta_xk(x,x)<1$) repeated reinforcement leads to the associative value $V(x)=w_xk(x,x)$ converging asymptotically and monotonically to the value $\lambda$. On the contrary, starting now with $V(x)>0$, repeated non-reinforcement leads to the associative value converging asymptotically and monotonically to $0$. Hence, repeated reinforcement is expected to increase the behavioral control of $x$ to a maximum, while repeated non-reinforcement is expected to diminish it up to null control on behavior.
Notably, in cases where more than one conditioned stimulus is considered, the interaction between different weights usually leads to interesting phenomena whose explanation is the purpose of these models \cite{WM97,ME02}.

It can be seen that the dynamics of these models can be captured by equation \eqref{eq}. Indeed, if $\mathbb{T}$ is composed of $n$ conditioned stimuli $x_1,\cdots,x_n$, then let $w(t)$ be the vector in $\mathbb{R}^n$ whose $i$-th component is the weight $w_{x_i}$ after the training event $t$. 
Such $w(t)$ satisfies equation \eqref{eq} if, given that some $x_i$ is being trained at time $t$, $F(t)$ is defined as the vector whose $i$-th component is $\delta_{x_i}\Lambda$ while the remaining entries are zero, and $A(t)$ is the matrix whose $i$-th row is $\langle \delta_{x_i}k(x_1,x_i),\delta_{x_i}k(x_2,x_i),\cdots, \delta_{x_i}k(x_n,x_i)\rangle$ while all the others rows are zero. Note that row $i$ of $F(t)$ could be zero (if training is of non-reinforcement type), but that the entry $(i,i)$ of $A(t)$ can be assumed non-zero since in the models we are interested $k(x_i,x_i)$ is always positive. Similarly, the entries of $A(t)$ can be assumed non-negative and smaller than $1$ as this is assumed true for $\delta_{x_j}k(x_i,x_j)$ in the models under consideration. This explains assumptions ii), iii) and iv). Moreover, under the piece-wise constant assumption on $A(t)$, it can be checked that assumption v) is satisfied with $D$ the diagonal matrix whose entries are $\delta_{x_i}^{-1}$ and $K=[k(x_i,x_j)]_{i,j}$ (which is symmetric and non-negative definite since $k$ is a kernel). Note also that the vector $V(t)$ containing in the $i$-th component the associative value of $x_i$ after training event $t$, satisfies $V(t)=Kw(t)$ (hence the psychological relevance of obtaining information about $Kw(t)$). 

Both the periodicity assumption and assumption i), which indicates the forcing vector $F(t)$ and the coefficient matrix $A(t)$ are piece-wise constant, have been imposed to capture the particularity of \quot{discrimination training}, a common application of these models \cite{P94,W03,WM97}. In this type of training, a set of conditioned stimuli $x_1,\cdots,x_n$ sharing some features are either reinforced or non-reinforced, consistently and repeatedly. Hence, the periodicity assumption together with assumption i) model the case of a discrimination training involving $n$ conditioned stimuli where, for each $i=1,\cdots, n$, stimulus $x_i$ is being either reinforced or non-reinforced consistently from time $T_{i-1}+1$ to time $T_i$, and where the whole training repeats after $T_n$ training events. The training from $(m-1)T_{n}+1$ to $mT_n$ can be considered the $m$-th \quot{training block}, and the interest on $w(mT_n)$ and $Kw(mT_n)$ derives from the fact that associative values are normally measured after some fixed number of training blocks. 

While organisms often show that, with enough training, they can \quot{solve} the discrimination (that is,  they have stronger behavioral responses to those conditioned stimuli being reinforced than to those that are non-reinforced) some models may fail to predict this success \cite{A24}. Hence, the ability to correctly predict the result of a discrimination training has been considered a benchmark for models \cite{WM97}. 

From a mathematical point of view, a model could be considered to successfully predict that a discrimination will be solved if it predicts that the associative values $V(t)=Kw(t)$ converge to a vector whose $i$-th entry is $\lambda$ if $x_i$ is being reinforced and $0$ otherwise. This means that, upon enough training, the model predicts that the conditioned stimuli being reinforced have a strong effect on behavior, while those non-reinforced have little or no effect on behavior. 

Despite the interest in studying asymptotic associative values under a discrimination training, researchers in the field have not considered the problem systematically, relying instead on computer simulations and ad-hoc assumptions to draw conclusions \cite[e.g.,][]{BVW00,P94,W03,A24}. Hence, the following work aims to provide a general treatment of the problem, in order to obtain conditions under which a model may or not predict that a discrimination will be solved and, more generally, obtain analytic expressions for the asymptotic associative values predicted under a discrimination training. 

To express the main result, let 
\[
\Sigma=\sum_{i=1}^nA(T_i) \quad , \quad F=\sum_{i=1}^nF(T_i).
\]
Note $A(T_i)$ and $F(T_i)$ are the matrix and vector reflecting training during the time period $[T_{i-1}+1,T_i]$ and, since these have disjoint non-zero rows, the entries of $\Sigma$ and $F$ contain all the relevant information related to the parameters of training. Hence, I call $F$ the \quot{training vector} and $´\Sigma$ the \quot{training matrix}. Note that, by assumption, $K=D\Sigma$.
Also, for any given matrix, its stable subspace is understood as the space spanned by the generalized eigenvectors of the matrix corresponding to eigenvalues $\mu$ such that $\left| \mu\right| < 1$. Conversely, its unstable subspace corresponds to the space generated by the generalized eigenvectors corresponding to eigenvalues $\mu$ such that $\left| \mu\right| \geq 1$.

\begin{theorem}\label{main_thm}
Let $a_{i,j}$ be the entries of $\Sigma$, and define $L$ as the $n\times n$ invertible lower triangular matrix with entries $l_{i,j}$ given by
\[
l_{i,j}=\begin{cases}
a_{i,j}& , j\leq i-1\\
\frac{a_{i,i}}{1-(1-a_{i,i})^{T_i-T_{i-1}}}& , j=i\\
0&, j\geq i+1.
\end{cases}
\]
\begin{enumerate}
\item If $\Sigma$ is invertible then
\[
\lim_{m\rightarrow\infty} w(mT_n)=\Sigma^{-1}F.
\]
\item Let $P$ be the projection whose image is the stable subspace of $I-L^{-1}\Sigma$ and whose kernel is the unstable subspace of $I-L^{-1}\Sigma$. Then
\[
\lim_{m\rightarrow\infty} \Sigma w(mT_n)=LPL^{-1}F.
\]
\end{enumerate}
\end{theorem}
Note that the condition that $\Sigma$ is invertible is really a condition on the values that the kernel takes on the conditioned stimuli since, by assumption, $K=D\Sigma$ and $D$ is invertible. Also, since $L$ combines information regarding the coefficients of the training matrix together with information about the actual length of training events (represented by the terms $T_i-T_{i-1}$), I call $L$ the \quot{effective training matrix}.

For all practical purposes, the above theorem completely achieves the main goal of this work. Indeed, it first states that if the training matrix $\Sigma$ is invertible then the weights (at the end of each training block) will converge to $\Sigma^{-1}F$, and so the associative values $V=Kw$ will converge to $DF$ which means the discrimination is correctly solved (since $DF$ has entry $i$ equal to $\lambda$ when the corresponding conditioned stimulus $x_i$ is of the reinforced type and $0$ otherwise). 
Secondly, when the training matrix $\Sigma$ is not invertible it is known that the weights may not converge \cite{A24}, but the theorem states that the associative values $Kw$ always do. In particular, in this last case, a model will predict the discrimination is solved if and only if the training is such that $L^{-1}F$ lies in the stable subspace of $I-L^{-1}\Sigma$.

As I show below, $P$ is the identity when $\Sigma$ is invertible. Hence, both statements of the theorem are summarized by saying that models predict convergence of associative values to $DLPL^{-1}F$. This means that, in the coordinate system given by the columns of $L$, the model dismisses those parts of the training vector $F$ which lie on $\ker P$. However, the psychological meaning of this result, in terms of stimulus processing and their effect on associative values and behavior, is not obvious at this point. For example, the entries of $K=D\Sigma$ are closely tied to how each model understand stimulus processing, but the conditions on a model which would guarantee that $K$ is invertible are not known. 
Understanding this issue could also help device experimental designs testing the accuracy of models. Indeed, according to the theorem, to have different predictions for different models it is necessary that at least one of the respective matrices $K$ is singular. 
Future work should attempt to explore these issues and, more broadly, discuss the general implications of Theorem \ref{main_thm} to associative learning.

While the main result of this work completely characterizes the predictions made by models that are describable as kernel machines, it does so under a hypothesis of periodic training. Since experimenters often choose to randomize the order in which training is carried out within each training block, the result may be said to apply to the idealized experimental design rather than to the actual experimental procedure.  Hence, it would be interesting to consider the possibility of introducing some randomness within each training block. However, as the periodicity assumption is fundamental, extending the result in this direction does not seem straightforward. This, together with the possibility of introducing other sources of randomness (such as a small additive noise for example), could be a relevant subject for future research. 

The proof of the main result begins by identifying that $y(m)=w(mT_n)$ satisfies the autonomous equation
\[
y(m)=(I-L^{-1}\Sigma)y(m-1)+L^{-1}F.
\]
Then, using that $K=D\Sigma$ is non-negative definite, it can be deduced that the eigenvalues $\mu$ of $I-L^{-1}\Sigma$ (the Floquet multipliers of the system) satisfy $\abs{\mu}\leq 1$, with $\abs{\mu}=1$ only when $\mu=1$. Moreover, $\mu=1$ is non-defective, and so the unstable subspace is associated exclusively with the multiplier $\mu=1$ and coincides with $\ker \Sigma$. This is the reason why applying the training matrix $\Sigma$ makes the weights converge. 

It should be remarked that, while Floquet theory is usually applied to homogeneous systems, this work shows that by looking at multiples of the period it is possible to extend its application to an inhomogenous problem such as \eqref{eq}. This idea has also been used in \cite{ST11}.

This work is organized as follows. Firstly, in Section \ref{derivations}, I derive the autonomous equation satisfied by $w(mT_n)$. Then, in Section \ref{training}, I find the expressions for the coefficient matrix and the forcing vector of the autonomous equation in terms of the training matrix $\Sigma$, the effective training matrix $L$, and the training vector $F$. Afterwards, in Section \ref{eigenvalues}, I carry out the study of Floquet multipliers and prove Theorem \ref{main_thm}. Finally, in Section \ref{application}, I contrast the results obtained here with the ones obtained in \cite{A24} where asymptotic associative values were found by explicitly solving the equations in the context of concrete models. 

\begin{remark}
 A model could be considered to successfully predict that a discrimination will be solved under a relaxed definition. For example, it may be enough that there exists positive constants $d_1<d_2$, with $d_2-d_1$ large enough, such that those conditioned stimuli being reinforced have asymptotic associative values larger than $d_2$ while those being non-reinforced have asymptotic associative values smaller than $d_1$. However, regardless of the definition put forward, by either the verification that $K$ is invertible or the computation of $LPL^{-1}F$, the theorem obtained here would allow researchers to determine if a model successfully predicts that a discrimination will be solved.
\end{remark}

\section{Derivation of the autonomous equation}\label{derivations}

I begin by deriving the equations for $w(mT_n)$ under the initial condition $w(0)=w_0$. This is done by first obtaining an explicit expression for the solution $w(t)$, and then using the periodicity of the coefficient matrix $A(t)$ and the forcing vector $F(t)$ to express $w(mT_n)$ in terms of $w((m-1)T_n)$. Given the periodicity of $A(t)$ and $F(t)$, the resulting equation has constant coefficient matrix and constant forcing vector. 

Since it is helpful here and later, let 
\[
A_i=A(T_i) \quad , \quad F(T_i)=f_i\bar{e}_i
\]
for some $f_i$ (here and throughout $\bar{e}_i$ is reserved for the standard unit vectors), and recall $a_{i,j}$ are the entries of the training matrix
\[
\Sigma=\sum_{i=1}^{n}A_{i}.
\]

Given the assumptions on $A_i$, it is easy to show that $I-A_i$ are invertible matrices.
\begin{proposition}
The eigenvalues of $I-A_i$ are $1$ and $1-a_{i,i}\in (0,1)$.
\end{proposition}
\begin{proof}
Since only row $i$ of $A_i$ is nonzero
\[
A_i\bar{e}_i=a_{i,i}\bar{e}_i
\]
while for $j\neq i$
\[
A_i\bar{e}_j=a_{i,j}\bar{e}_i=\frac{a_{i,j}}{a_{i,i}}A_i\bar{e}_i \Rightarrow A_i(\bar{e}_j-\frac{a_{i,j}}{a_{i,i}}\bar{e}_i)=0.
\]
Hence, $A_i$ has two eigenvalues: $a_{i,i}$ (with geometric multiplicity $1$) and $0$ (with geometric multiplicity $n-1$).
Thus, to finish the proof, it suffices to note that $\mu$ is eigenvalue of $A_i$ if and only if $1-\mu$ is eigenvalue of $I-A_i$.
\end{proof}

With the above proposition, it follows the homogeneous problem has fundamental matrix
\[
\Phi(t)=\prod_{s=1}^{t}(I-A(s))
\]
where the product is always understood as written from right to left and the empty product is defined as $1$. Then, the solution of the initial value problem is 
\[
w(t)=\Phi(t)w_0 + \Phi(t)\sum_{s=1}^{t}\Phi^{-1}(s)F(s).
\]

To continue, for $i=1,\cdots,n$ let $L_{i}=T_{i}-T_{i-1}$ be the length of the training events from $T_{i-1}+1$ to $T_{i}$, and let
\[
M=\prod_{s=1}^n(I-A_s)^{L_s}=\Phi(T_n).
\]
This matrix is often called the \quot{monodromy matrix} in Floquet theory.

Then, a somewhat lengthy but simple computation shows that
\begin{align*}
w(mT_{n})&= \left[ \prod_{s=1}^{mT_n}(I-A(s))\right] w_0+ \left[ \prod_{s=1}^{mT_n}(I-A(s))\right]\sum_{s=1}^{mT_{n}}\Phi(s)^{-1} F(s)\\
	&= Mw((m-1)T_n)+\sum_{s=1+(m-1)T_n}^{mT_{n}}\left[ \prod_{i=s+1}^{mT_n}(I-A(i))\right] F(s) \\
	&= Mw((m-1)T_n)+ \sum_{s=1}^{T_{n}}\left[ \prod_{i=s+1}^{T_n}(I-A(i))\right] F(s) \\
	&= Mw((m-1)T_n)+\sum_{l=1}^nf_l\frac{1-(1-a_{l,l})^{L_l}}{a_{l,l}}\left[ \prod_{i=l+1}^{n}(I-A_i)^{L_i}\right]\bar{e}_l .
\end{align*}
That is, $y(m)=w(mT_n)$ satisfies the equation
\[
y(m)= My(m-1)+b
\]
where
\[
b=\sum_{l=1}^nf_l\frac{1-(1-a_{l,l})^{L_l}}{a_{l,l}}\left[ \prod_{i=l+1}^{n}(I-A_i)^{L_i}\right]\bar{e}_l .
\]
In particular, 
\[
y(m)=M^{m}w_0+\sum_{s=0}^{m-1}M^{s}b.
\]

To continue, I turn to the study of the monodromy matrix $M$ and the vector $b$.

\section{Study of $M$ and $b$}\label{training}

The aim of this section is to express the monodromy matrix $M$ and the vector $b$ in terms of the training matrix $\Sigma$, the training vector $F$, and the effective training matrix $L$.

\subsection{Study of $M$}
The first objective is to relate the monodromy matrix $M$ to the training matrix $\Sigma$. By considering the structure of the matrices $(I-A_i)^{L_i}$ (which are related to $\Sigma$) it is possible to describe them as a product of $n$ elemental matrices. Then, $M$ itself can be described as a product of $n^2$ elemental matrices, from where a representation of the entries of $M$ can be deduced. Since $I-M$ only significantly changes on the diagonal, it is then possible to find row operations that reduce $I-M$ to $\Sigma$. The effective training matrix $L$ appears then as the matrix representing these row operations. 

Now, recall
\[
M=\prod_{i=1}^n(I-A_i)^{L_i} \quad , \quad \Sigma=\sum_{i=1}^n A_i.
\]
Also, since it is useful here and later, let
\[
\gamma_i=(1-a_{i,i})^{L_i} \quad , \quad r_i=\frac{1-(1-a_{i,i})^{L_i}}{a_{i,i}}=\frac{1-\gamma_i}{a_{i,i}}.
\]
Note that the effective training matrix $L=[l_{i,j}]$ defined on the introduction has
\[
l_{i,j}=\begin{cases}
a_{i,j}& , j\leq i-1\\
r_i^{-1}& , j=i\\
0&, j\geq i+1.
\end{cases}
\]

\begin{theorem}
$I-M=L^{-1}\Sigma$.
\end{theorem}
\begin{proof}
Observe that row $i\neq j$ of $(I-A_j)^{L_j}$ is $\bar{e}_i^{T}$ while row $j$ has entries $(j,k)$ given by
\[
\begin{cases}
-a_{j,k}\sum_{s=0}^{L_j-1}(1-a_{j,j})^s& ,k\neq j\\
(1-a_{j,j})^{L_j}& ,k=j
\end{cases}=
\begin{cases}
-a_{j,k}r_j& ,k\neq j\\
\gamma_j & ,k=j
\end{cases}
\]
Then, each of $(I-A_j)^{L_j}$ is the product of $n$ elemental matrices. Indeed
\[
E^{j}_n\cdots E^{j}_1 I=(I-A_j)^{L_j}
\]
where $E^{j}_1$ is the matrix representing multiplication of the $j$-th row by $\gamma_j$, while the remaining matrices represent, for $k\neq j$, addition to row $j$ of $-a_{j,k}r_j$ times row $k$ (the order of these remaining $n-1$ operations is immaterial). In particular, all the $E^j_k$ act only on row $j$. 
Hence, if $M$ has entries $m_{i,j}$, from 
\[
M=E^{n}_n\cdots E^{n}_1 \cdots E^{1}_n\cdots E^{1}_1 I
\]
can be deduced the recursive formulas
\[
m_{i,j}=\begin{cases}
-\sum_{k=1}^{i-1}a_{i,k}r_im_{k,j} & , j\leq i-1 \\
\gamma_i -\sum_{k=1}^{i-1}a_{i,k}r_im_{k,i} &, j=i\\
-\sum_{k=1}^{i-1}a_{i,k}r_im_{k,j} -a_{i,j}r_i  &, j\geq i+1.
\end{cases}
\]

I now show there exist a sequence of row operations that reduce $I-M$ to $\Sigma$. Indeed, first note that $I-M$ has entries $\tilde{m}_{i,j}$ that satisfy
\[
\tilde{m}_{i,j}=\begin{cases}
\sum_{k=1}^{i-1}a_{i,k}r_im_{k,j} & , j\leq i-1 \\
1-\gamma_i +\sum_{k=1}^{i-1}a_{i,k}r_im_{k,i} &, j=i\\
\sum_{k=1}^{i-1}a_{i,k}r_im_{k,j} + a_{i,j}r_i  &, j\geq i+1
\end{cases}
\]
while also $\tilde{m}_{i,j}=-m_{i,j}$ for $i\neq j$ and $\tilde{m}_{i,i}=1-m_{i,i}$.
Then, for $s=n$ to $s=1$ consider acting on row $s$ of $I-M$ by adding $a_{s,k}r_s$ times row $k$, for all $k\leq s-1$, and then by multiplying row $s$ by $a_{s,s}(1-\gamma_{s})^{-1}$. The entries $c_{i,j}$ of the resulting matrix satisfy
\begin{align*}
c_{i,j}&=\frac{a_{i,i}}{1-\gamma_i}\left( \tilde{m}_{i,j}+\sum_{k=1}^{i-1}a_{i,k}r_i\tilde{m}_{k,j}\right) \\
&=\begin{cases}
\frac{a_{i,i}}{1-\gamma_i}\left( a_{i,j}r_im_{j,j}+a_{i,j}r_i(1-m_{j,j})\right) &, j\leq i-1\\
a_{i,i} &, j=i\\
a_{i,j} &, j\geq i+1
\end{cases} \\
&=
\begin{cases}
a_{i,j}&, j\leq i-1\\
a_{i,i} &, j=i\\
a_{i,j} &, j\geq i+1
\end{cases}
\end{align*}
which shows $I-M$ has been reduced to $\Sigma$ by multiplying on the left by the matrix
\begin{align*}
L&=\begin{bmatrix}
\frac{a_{1,1}}{1-\gamma_1} & 0 & 0 & \cdots & 0\\
\frac{a_{2,1}r_2a_{2,2}}{1-\gamma_2}&  \frac{a_{2,2}}{1-\gamma_2}& 0 & \cdots & 0\\
\frac{a_{3,1}r_3a_{3,3}}{1-\gamma_3}& \frac{a_{3,2}r_3a_{3,3}}{1-\gamma_3} & \frac{a_{3,3}}{1-\gamma_3}& \cdots & 0\\
\vdots & \vdots & \vdots & \ddots & \vdots\\
\frac{a_{n,1}r_na_{n,n}}{1-\gamma_n}& \frac{a_{n,2}r_na_{n,n}}{1-\gamma_n} & \frac{a_{n,3}r_na_{n,n}}{1-\gamma_n} & \cdots & \frac{a_{n,n}}{1-\gamma_n}
\end{bmatrix}\\
&=\begin{bmatrix}
r_1^{-1} & 0 & 0 & \cdots & 0\\
a_{2,1}&  r_2^{-1}& 0 & \cdots & 0\\
a_{3,1}& a_{3,2} & r_3^{-1}& \cdots & 0\\
\vdots & \vdots & \vdots & \ddots & \vdots\\
a_{n,1}& a_{n,2} & a_{n,3} & \cdots & r_n^{-1}
\end{bmatrix}
\end{align*}
as claimed. 
\end{proof}

Since $I-M=L^{-1}\Sigma$, and $L^{-1}$ is invertible, the next corollary follows at once.

\begin{corollary}
$I-M$ is invertible if and only if $\Sigma$ is invertible.
\end{corollary}

\subsection{Study of $b$}

I now study the vector $b$. The idea is to look at $Lb$ which, given the structure of $b$, means considering the $(k-1)$-th column of the product between the effective training matrix $L$ and $\prod_{s=k}^{n}(I-A_s)^{L_s}$. Using again the special structure of the matrices $(I-A_k)^{L_k}$, and their relationship with $L$, it is possible to use induction on $k$ to keep track of the $(k-1)$-th column. This shows said column is simply a multiple of the standard unit vector $\bar{e}_{k-1}$.

Recall 
\[
b=\sum_{l=1}^nf_l\frac{1-(1-a_{l,l})^{L_l}}{a_{l,l}}\left[ \prod_{i=l+1}^{n}(I-A_i)^{L_i}\right]\bar{e}_l 
\]
and
\[
F=\sum_{i=1}^nF(T_i)=\sum_{i=1}^n f_i\bar{e}_i.
\] 

\begin{theorem}
$Lb=F$.
\end{theorem}
\begin{proof}
Recall that row $i\neq j$ of $(1-A_j)^{L_j}$ is $\bar{e}_i^{T}$, and that row $j$ has entries $(j,k)$ given by
\[
\begin{cases}
-a_{j,k}r_j& ,k\neq j\\
\gamma_j & ,k=j.
\end{cases}
\]

The key of the proof is showing, for $k=2,\cdots,n+1$ the block diagonal structure
\[
L\prod_{s=k}^n(I-A_s)^{L_s}=\left[\begin{array}{c | c}
[L]_{k-1} &0 \\ \hline
0& B
\end{array}\right]
\]
where $[L]_{k-1}$ is the $k-1\times k-1$ matrix whose entries are as the entries $l_{i,j}$ of the matrix $L$ for $i,j\leq k-1$, and $B$ is some $(n-k+1)\times (n-k+1)$ matrix (note that when $k=n+1$ the product is simply the matrix $L$). Since the $(k-1)$-th column of the above matrix is $(r_{k-1})^{-1}\bar{e}_{k-1}$, it is immediately deduced that
\[
Lb=\sum_{l=1}^nf_lr_lL\left[ \prod_{i=l+1}^{n}(I-A_i)^{L_i}\right]\bar{e}_l =\sum_{l=1}^nf_lr_lr_l^{-1}\bar{e}_l=F.
\]

Now, the block structure is true for $k=n+1$, and if it is true for $k=k_0+1\in [3, n+1]$ then
\begin{multline*}
L\prod_{s=k_{0}}^n(I-A_s)^{L_s}=\left[\begin{array}{c | c}
[L]_{k_0} &0 \\ \hline
0& B
\end{array}\right](I-A_{k_0})^{L_{k_0}}\\=\left[\begin{array}{c | c}
[L]_{k_0} &0 \\ \hline
0& B
\end{array}\right]
\left[\begin{array}{c | c | c}
I_{k_0-1} &0 & 0\\ \hline
-a_{k_0,1}r_{k_0}\cdots & \gamma_{k_0} & \cdots -a_{k_0,n}r_{k_0}\\ \hline
0 & 0 & I_{n-k_0}
\end{array}\right]\\
=\left[\begin{array}{c | c | c}
[L]_{k_0-1} &0 & 0\\ \hline
c_{k_0,1}\cdots & c_{k_0,k_0} & \cdots c_{k_0,n}\\ \hline
0 & 0 & B
\end{array}\right]
\end{multline*}
where $I_j$ is the $j\times j$ identity matrix, and the entries $c_{k_0,j}$ can be computed for $j< k_0$ as
\[
c_{k_0,j}=\left( r_{k_0}^{-1}\bar{e}_{k_0}+\sum_{i=1}^{k_0-1}a_{k_0,i}\bar{e}_i\right) \cdot \left( -a_{k_0,j}r_{k_0}\bar{e}_{k_0}+\bar{e}_j\right) =-a_{k_0,j}+a_{k_0,j}=0.
\]
Hence
\[
L\prod_{s=k_{0}}^n(I-A_s)^{L_s} =\left[\begin{array}{c | c | c}
[L]_{k_0-1} &0 & 0\\ \hline
0 & c_{k_0,k_0} & \cdots c_{k_0,n}\\ \hline
0 & 0 & B
\end{array}\right]
\]
and the result holds for $k=k_0$. This finishes the proof. 
\end{proof}

\section{Asymptotic values}\label{eigenvalues}

I have shown the weights $y(m)=w(mT_n)$ evolve according to
\[
y(m)=My(m-1)+b \Leftrightarrow y(m)=M^{m}w(0)+\sum_{s=0}^{m-1}M^{s}b.
\]
where
\[
I-M=L^{-1}\Sigma \quad , \quad b=L^{-1}F.
\]

The behavior of the system is tied to the eigenvalues of the monodromy matrix $M$, which are the Floquet multipliers of the system. Recall that for some diagonal matrix $D$ with positive entries, $K=D\Sigma$ is symmetric and non-negative definite.

\begin{theorem}
If $\mu$ is an eigenvalue of $M$ then $\abs{\mu}\leq 1$. Moreover, if $\abs{\mu}=1$ then $\mu=1$, and this is a non-defective eigenvalue whose corresponding eigenspace coincides with $\ker(\Sigma)$.
\end{theorem}
\begin{proof}
First, note $\mu$ is an eigenvalue of $M$ if and only if $\eta=1-\mu$ is an eigenvalue of $L^{-1}\Sigma$ and that
\[
L^{-1}\Sigma y=\eta y \Leftrightarrow \Sigma y=\eta Ly \Leftrightarrow K y=\eta DL y.
\]
Also, recall that $L$ is lower triangular and that the entries below the diagonal are the same as those of $\Sigma$. Hence, if $k_{i,j}$ are the entries of $K$, and $x_i$ are the components of $x\in \R^n$,
\[
x^TKx=\sum_{i,j=1}^nk_{i,j}x_ix_j=\sum_{i=1}^n k_{i,i}x_i^2+2\sum_{j<i}k_{i,j}x_ix_j
\]
and so, if $d_i$ are the diagonal entries of $D$, 
\begin{equation}\label{main_id}
x^TDLx=\sum_{j<i}k_{i,j}x_ix_j+\sum_{i=1}^n d_ir_i^{-1}x_i^2=\frac{1}{2}x^TKx +\sum_{i=1}^n (d_ir_i^{-1}-\frac{1}{2}k_{i,i})x_i^2.
\end{equation}
Note $d_ir_i^{-1}-\frac{1}{2}k_{i,i}$ are positive since $d_i$ and $r_i^{-1}-\frac{1}{2}a_{i,i}$ are positive. To ease notation, define
\[
c(x)=\sum_{i=1}^n (d_ir_i^{-1}-\frac{1}{2}k_{i,i})x_i^2
\]
and observe $c(x)$ is positive if $x\neq 0$.

Now, let $\mu =a+bi$ be an eigenvalue of $M$ with eigenvector $z=u+iv$, $a,b\in \R$, $u,v\in \R^n$, and assume first $b\neq 0$.
From $Kz=(1-\mu)DLz$ follows that
\begin{align*}
Ku&=(1-a)DLu+bDLv\\
Kv&=(1-a)DLv-bDLu
\end{align*}
so that multiplying each equation by $u^T$ and $v^T$, and using that $u^TKv=v^TKu$, it can be deduced that
\begin{multline*}
\frac{1-a}{b}\left[ u^TKu-(1-a)u^TDLu\right] -bu^TDLu\\=-\frac{1-a}{b}\left[ v^TKv-(1-a)v^TDLv\right] +bv^TDLv.
\end{multline*}
Then, using equation \eqref{main_id},
\begin{multline*}
\left[ \frac{1-a}{b}-\frac{(1-a)^2+b^2}{2b}\right] u^TKu-\left[ \frac{(1-a)^2+b^2}{b}\right] c(u)
\\=
-\left[ \frac{1-a}{b}-\frac{(1-a)^2+b^2}{2b}\right] v^TKv+\left[ \frac{(1-a)^2+b^2}{b}\right] c(v)
\end{multline*}
so
\[
\left[ \frac{1-a}{b}-\frac{(1-a)^2+b^2}{2b}\right]\left( u^TKu+v^TKv\right) =\left[ \frac{(1-a)^2+b^2}{b}\right]\left( c(v)+c(u)\right) 
\]
or
\[
\frac{1-(a^2+b^2)}{2}\left( u^TKu+v^TKv\right) =((1-a)^2+b^2)\left( c(v)+c(u)\right). 
\]
Since $b\neq 0$, the right hand side is strictly positive. Hence, since $u^TKu+v^TKv$ is non-negative, there follows $\abs{\mu}< 1$. 
In the case $b=0$ it may be assumed $z$ is real and a simple computation gives
\[
\frac{1+a}{2}z^TKz=(1-a)c(z),
\]
from where $\mu=a\in (-1,1]$. Moreover, if $\mu=1$ then $Kz=0$ so $z\in \ker(\Sigma)$ since $D$ is invertible.

Finally, I show that if $\mu=1$ is an eigenvalue of $M$ then it is non-defective, so the corresponding generalized eigenspace is simply $\ker(I-M)=\ker(\Sigma)$.
To this end, it suffices to show $(M-I)^2y=0$ implies $(M-I)y=0$. Hence, let $(I-M)y=x$ with $(I-M)x=0$. Since $I-M=L^{-1}\Sigma$ it follows that
\[
 Kx=0 \quad , \quad Ky=DLx
\]
and so
\[
x^TKy=x^TDLx=\frac{1}{2}x^TKx+c(x).
\]
But both $x^TKy$ and $x^TKx$ are zero, and so $x=0$. This concludes the proof.
\end{proof}

The main result can now be proven.

\begin{proof}[Proof of Theorem \ref{main_thm}]
It suffices to prove the second statement since then the first statement follows from the fact that if $\Sigma$ is invertible then $\ker \Sigma$ is trivial and so $P=I$.

Recall 
\[
y(m)=M^mw(0)+\sum_{s=0}^{m-1}M^{s}b.
\]
and note
\[
L^{-1}\Sigma \sum_{s=0}^{m-1}M^{s}=(I-M) \sum_{s=0}^{m-1}M^s=(I-M^{m}).
\]
so that
\[
\Sigma y(m)=\Sigma M^m w(0)+L(I-M^m)b.
\]
Using that $M^m P x\rightarrow 0$, $M^m(I-P)x=(I-P)x$, and $\Sigma(I-P)x=0$ (since $(I-P)x$ lies in $\ker \Sigma$) it follows that
\[
\Sigma y(m)=M^m P w(0)+Lb-L(M^mPb+(I-P)b)\rightarrow Lb-L(I-P)b=LPb.
\]
\end{proof}

Before moving onto the final section, I briefly discuss how an explicit expression for $P$ can be obtained. Suppose $r(\mu)$ is the minimal polynomial of $M=I-L^{-1}\Sigma$ and let $q(\mu)=r(\mu)/(\mu-1)$. Since $\mu=1$ is a non-defective eigenvalue, and $(I-P)x$ belongs to the corresponding eigenspace, it follows that 
\[
q(M)(I-P)x=q(1)(I-P)x \quad , \quad q(M)Px=0
\]
and so
\[
(I-P)x=\frac{q(M)}{q(1)}x \Leftrightarrow  Px=\left( I-\frac{q(M)}{q(1)}\right) x.
\]

\section{An application}\label{application}

In this section I contrast the theory developed so far with results obtained in \cite{A24} regarding the discrimination training known as \quot{biconditional discrimination} \cite{S75}. In this experimental design, four conditioned stimuli $A$, $B$, $X$, $Y$ are combined into the four compounded conditioned stimuli $AY$, $AX$, $BY$, $BX$, and experimental subjects are consistently and repeatedly trained in training blocks where, consecutively, $AY$ is reinforced, $AX$ is non-reinforced, $BY$ is non-reinforced, and $BX$ is reinforced.
There is evidence that organisms can solve the discrimination, i.e., they show greater behavioral responses to the $AY$ and $BX$ compounds than to the $AX$ and $BY$ compounds \cite{S75}. From the point of view of models, this would mean the associative values of $AY$ and $BX$ should tend to $\lambda$, while the associative values of $AX$ and $BY$ should tend to $0$.

In \cite{A24}, biconditional discrimination was studied under some simplifying assumptions regarding the learning rate parameters and in the context of three different models of associative learning: Rescorla-Wagner's model (RWM), Pearce's model (P94), and the Inhibited Elements Models (IEM) (see the cited work for a description of these models). By explicitly solving the equations, and computing the respective limits, it was shown that P94 and IEM successfully predict that the discrimination will be solved, while RWM fails to make such a prediction. Indeed, it was obtained that both P94 and IEM predict that, at the end of each training block, the associative values of $AY$ and $BX$ converge to $\lambda$ while the associative values of $AX$ and $BY$ converge to $0$. In contrast, it was shown that RWM predicts the associative values of $AX$ and $BY$ converge to $\lambda/2$ while the associative values of $AY$ and $BX$ converge to multiples of $\lambda$ with coefficients that depend on the values of the parameters of learning. 

The objective here is to show that the same results can be obtained by means of Theorem \ref{main_thm}. Moreover, a comparison with the method used in \cite{A24} highlights at least three advantages of the method proposed here. Firstly, the correct prediction made by P94 and IEM is obtained here by showing that the corresponding matrices $K$ are invertible, which contrasts in its simplicity to the process of solving the associated equations to determine the limits. Moreover, since solving the equations is a complicated process, the method in \citep{A24} imposed the assumption of uniform learning rate parameters, an assumption not needed here to obtain the desired conclusion regarding P94 an IEM. Lastly, while the assumption of uniform learning rate parameters is used here to study RWM, the computations performed are much shorter and straightforward than the ones performed in \citep{A24}.

Now, let $x_1=AY$, $x_2=AX$, $x_3=BY$, and $x_4=BX$, so that the training vector is
\[
F=\begin{bmatrix}
\delta_1\lambda\\
0\\
0\\
\delta_4 \lambda
\end{bmatrix}.
\]
In the models under consideration, computations depends on learning rate parameters $\beta_{+},\beta_{-}\in (0,1)$ associated to reinforcement and non-reinforcement, respectively. Also, if $Z$ is any of the four single stimuli or the four compounds, then $Z$ has its own learning rate parameter $\alpha_Z$. 
Then, for both RWM and IEM the parameters to be used are $\delta_1=\delta_4=\beta_+$ and $\delta_2=\delta_3=\beta_{-}$, while for P94 they are $\delta_1=\beta_+\alpha_{AY}$, $\delta_2=\beta_-\alpha_{AX}$, $\delta_3=\beta_-\alpha_{BY}$, and $\delta_4=\beta_+\alpha_{BX}$.

Computation of the corresponding kernels (see \cite{A24} for details) leads to 
\[
K_{RWM}=\begin{bmatrix}
\alpha_A+\alpha_Y & \alpha_A 		& \alpha_Y  			& 0\\
\alpha_A		& \alpha_A+\alpha_X & 0					 & \alpha_X\\
\alpha_Y			& 0 			& \alpha_B+\alpha_Y & \alpha_B\\
0			& \alpha_X					& \alpha_B 			& \alpha_B+\alpha_X
\end{bmatrix}
\]
\[
K_{IEM}=\begin{bmatrix}
\frac{\alpha_A+\alpha_Y}{2} & \frac{\alpha_A}{3} & \frac{\alpha_Y}{3} & 0\\
\frac{\alpha_A}{3} & \frac{\alpha_A+\alpha_X}{2} & 0 & \frac{\alpha_X}{3}\\
\frac{\alpha_Y}{3} & 0 & \frac{\alpha_B+\alpha_Y}{2} & \frac{\alpha_B}{3}\\
0 & \frac{\alpha_X}{3} & \frac{\alpha_B}{3} & \frac{\alpha_B+\alpha_X}{2}
\end{bmatrix}
\]
\[
K_{P94}=\begin{bmatrix}
1 & 1/4 & 1/4 & 0\\
1/4 & 1 & 0 & 1/4\\
1/4 & 0 & 1 & 1/4\\
0 & 1/4 & 1/4 & 1
\end{bmatrix}
\]
as corresponding matrix $K$ for RWM, IEM, and P94, respectively. 
It can be verified that $K_{IEM}$ and $K_{P94}$ are invertible, so both IEM and P94 predict that the discrimination will be solved. 
On the other hand, $K_{RWM}$ is not invertible (its kernel is generated by $\langle 1, -1 , -1 , 1\rangle$). 

In order to find the asymptotic associative values predicted by RWM, while keeping the computations tractable, I make the assumption of uniform learning rate parameters $\beta_-=\beta_+=\beta$ and $\alpha_A=\alpha_B=\alpha_X=\alpha_Y=\alpha$ that was used in \cite{A24}.
Under this simplifying assumption, and writing $\gamma=\alpha\beta$ for simplicity,
\[
\Sigma=
\begin{bmatrix}
2\gamma & \gamma		& \gamma 			& 0\\
\gamma	& 2\gamma & 0					 & \gamma\\
\gamma	& 0 			& 2\gamma & \gamma\\
0			& \gamma				& \gamma		& 2\gamma
\end{bmatrix} \ , \ 
L=\begin{bmatrix}
1 & 0 & 0 & 0 \\
\gamma& 1 & 0					 & 0\\
\gamma		& 0 			& 1 & 0\\
0			& \gamma				& \gamma			& 1
\end{bmatrix} \ , \
L^{-1}=\begin{bmatrix}
1 & 0 & 0 & 0 \\
-\gamma		& 1 & 0					 & 0\\
-\gamma			& 0 			& 1 & 0\\
2\gamma^2		& -\gamma				& -\gamma			& 1
\end{bmatrix}
\]
and so
\[
M=\begin{bmatrix}
1-2\gamma & -\gamma & - \gamma & 0\\
-\gamma(1-2\gamma) & 1-2\gamma+\gamma^2 & \gamma^2 & -\gamma\\
-\gamma(1-2\gamma) &\gamma^2 & 1-2\gamma+\gamma^2 & - \gamma\\
-2\gamma^2(-1+2\gamma)&-\gamma(1-2\gamma+2\gamma^2)&-\gamma(1-2\gamma+2\gamma^2)&1-2\gamma+2\gamma^2
\end{bmatrix}.
\]
This matrix has eigenvalues $1$, $1-2\gamma$, and $(1-2\gamma)^2$, the second one being of multiplicity $2$ and non-degenerate. Hence, the predicted asymptotic associative values are
\[
V=\frac{1}{\beta}L P L^{-1}F
\]
where, by the comments at the end of the previous section,
\[
P=I-\frac{(M-(1-2\gamma)I)(M-(1-2\gamma)^2)}{(1-(1-2\gamma))(1-(1-2\gamma)^2)}.
\]
Since 
\[
(M-(1-2\gamma)I)(M-(1-2\gamma)^2)=\gamma^2\begin{bmatrix}
2(1-2\gamma)& -2(1-\gamma) & -2(1-\gamma) & 2\\
-2(1-2\gamma)& 2(1-\gamma) & 2(1-\gamma) & -2\\
-2(1-2\gamma)& 2(1-\gamma) & 2(1-\gamma) & -2\\
2(1-2\gamma)& -2(1-\gamma) & -2(1-\gamma) & 2
\end{bmatrix}
\]
it follows
\[
P=\begin{bmatrix}
\frac{3-2\gamma}{4(1-\gamma)}& 1/4 & 1/4 & -\frac{1}{4(1-\gamma)}\\
\frac{1-2\gamma}{4(1-\gamma)}& 3/4 & -1/4 & \frac{1}{4(1-\gamma)}\\
\frac{1-2\gamma}{4(1-\gamma)}& -1/4 & 3/4 & \frac{1}{4(1-\gamma)}\\
-\frac{1-2\gamma}{4(1-\gamma)}& 1/4 & 1/4 & \frac{3-4\gamma}{4(1-\gamma)}
\end{bmatrix}
\]
and so
\[
V=\lambda\begin{bmatrix}
\frac{3-4\gamma}{4(1-\gamma)}& \frac{1}{4(1-\gamma)} & \frac{1}{4(1-\gamma)} & -\frac{1}{4(1-\gamma)}\\
1/4& 3/4 & -1/4 & 1/4\\
1/4& -1/4 & 3/4 & 1/4\\
-\frac{1-2\gamma}{4(1-\gamma)}& \frac{1-2\gamma}{4(1-\gamma)} & \frac{1-2\gamma}{4(1-\gamma)} & \frac{3-2\gamma}{4(1-\gamma)}
\end{bmatrix}
\begin{bmatrix}
1\\0\\0\\1
\end{bmatrix}=\lambda\begin{bmatrix}
\frac{1-2\gamma}{2(1-\gamma)}\\
1/2\\
1/2\\
\frac{1}{2(1-\gamma)}
\end{bmatrix}.
\]
These are the exact values found in \cite{A24}.

Figure 1 shows, in a concrete example, the evolution of associative values for $AY$, $AX$, $BY$, and $BX$, according to the models P94, IEM, and RWM. As expected, the plots for P94 and IEM show the associative values of reinforced compounds converge to maximal associative value, while the associative values of non-reinforced compounds converge to zero. Similarly, the plot for RWM shows the model cannot account for the discrimination, and that the associative values converge to the predicted values $V(AY)=1/3$, $V(AX)=V(AY)=1/2$, and $V(BX)=2/3$, given the choices $\lambda=1$ and $\alpha=\beta=1/2$ used in the simulation. 

\begin{figure}[h]
	\begin{center}
		\includegraphics*[width=0.49\textwidth]{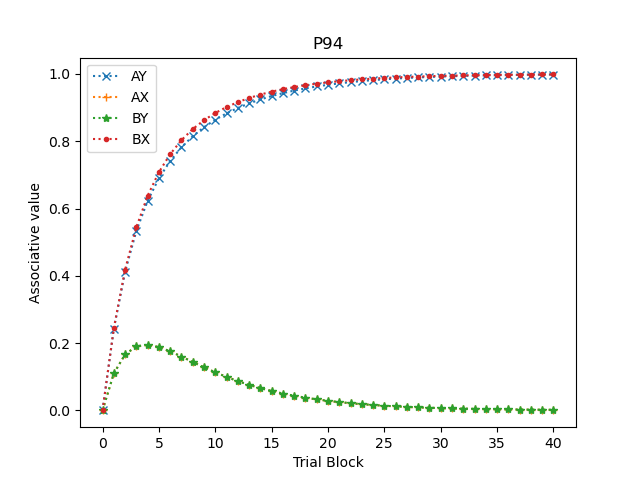}
		\includegraphics*[width=0.49\textwidth]{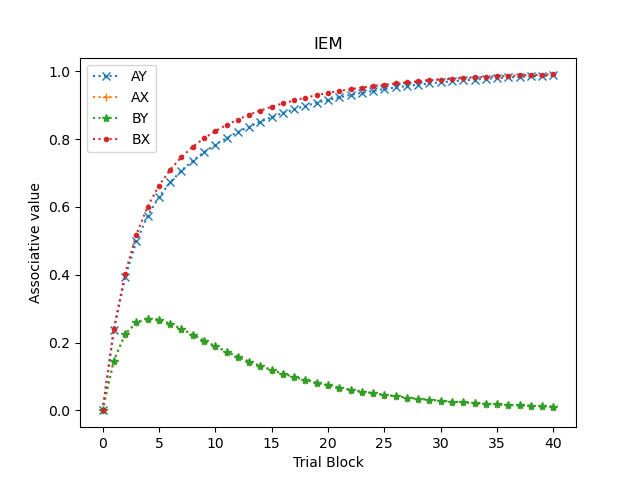}
		
		\includegraphics*[width=0.5\textwidth]{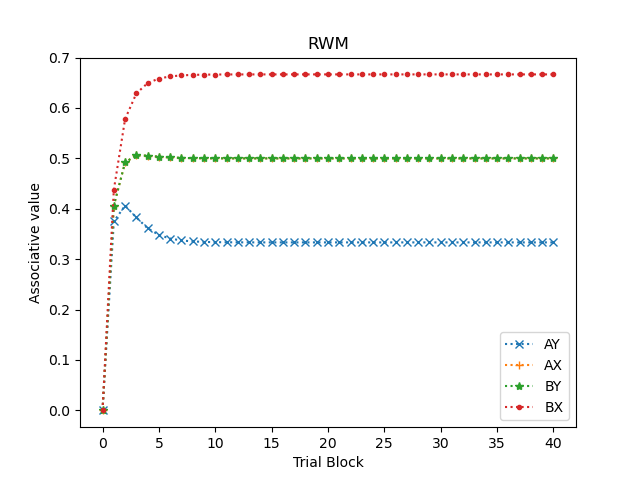}
	\end{center}
	\caption{evolution of associative values for $AY$, $AX$, $BY$, and $BX$, under a biconditional discrimination according to the models P94, IEM, and RWM. The plots show the evolution of associative values $V=Kw$, where $w$ obeys equation \eqref{eq}, across $40$ blocks of training. Uniform parameters $\lambda=1$, $\beta=0.5$, and $\alpha=0.5$ where chosen in all simulations. Note in all cases the curves for $AX$ and $BY$ are nearly indistinguishable.}
	\label{figure_1}
\end{figure}

\section*{Disclosure statement}

The authors report there are no competing interests to declare.

\section*{Funding}

This research did not receive any specific grant from funding agencies in the public,
commercial, or not-for-profit sectors.

\bibliographystyle{tfs}
\bibliography{biblio}

@article{NNS08,
author = {T. Naito and P. H. A. Ngoc and J. S. Shin},
title = {{Floquet representations and asymptotic behavior of solutions to periodic linear difference equations}},
volume = {38},
journal = {Hiroshima Mathematical Journal},
number = {1},
publisher = {Hiroshima University, Mathematics Program},
pages = {135 -- 154},
year = {2008},
doi = {10.32917/hmj/1207580348},
URL = {https://doi.org/10.32917/hmj/1207580348}
}

@ARTICLE{LY97,
  author={Lindquist, A. and Yakubovich, V.A.},
  journal={IEEE Transactions on Automatic Control}, 
  title={Optimal damping of forced oscillations in discrete-time systems}, 
  year={1997},
  volume={42},
  number={6},
  pages={786-802},
  doi={10.1109/9.587333}}

@article{DT98,
title = {Stabilization of Linear Discrete-Time Periodic Systems},
journal = {IFAC Proceedings Volumes},
volume = {31},
number = {18},
pages = {485-490},
year = {1998},
note = {5th IFAC Conference on System Structure and Control 1998 (SSC'98), Nantes, France, 8-10 July},
issn = {1474-6670},
doi = {https://doi.org/10.1016/S1474-6670(17)42038-6},
url = {https://www.sciencedirect.com/science/article/pii/S1474667017420386},
author = {Carlos E. {de Souza} and Alexandre Trofino}
}

@article{K22,
  title={Dispersal-induced growth in a time-periodic environment},
  author={Katriel, Guy},
  journal={Journal of Mathematical Biology},
  volume={85},
  number={3},
  pages={24},
  year={2022},
  publisher={Springer}
}

@ARTICLE{SGMNP16,
  author={Saublet, Louis-Marie and Gavagsaz-Ghoachani, Roghayeh and Martin, Jean-Philippe and Nahid-Mobarakeh, Babak and Pierfederici, Serge},
  journal={IEEE Transactions on Industrial Electronics}, 
  title={Asymptotic Stability Analysis of the Limit Cycle of a Cascaded DC–DC Converter Using Sampled Discrete-Time Modeling}, 
  year={2016},
  volume={63},
  number={4},
  pages={2477-2487},
  doi={10.1109/TIE.2015.2509908}
}

@article{JSW07,
  title={A tutorial on kernel methods for categorization},
  author={J{\"a}kel, Frank and Sch{\"o}lkopf, Bernhard and Wichmann, Felix A},
  journal={Journal of Mathematical Psychology},
  volume={51},
  number={6},
  pages={343--358},
  year={2007},
  publisher={Elsevier}
}

@book{Pavlov1927,
  title={Conditioned reflexes},
  author={Pavlov, I. P.},
  year={1927},
  publisher={Oxford university press}
}

@article{A24,
  title={On the mathematical formalization of the Inhibited Elements Model},
  author={Aguirre, Natham},
  journal={Journal of Mathematical Psychology},
  volume={123},
  pages={102887},
  year={2024},
  publisher={Elsevier}
}

@article{S75,
  title={Pavlovian compound conditioning in the rabbit},
  author={Saavedra, Maria A},
  journal={Learning and Motivation},
  volume={6},
  number={3},
  pages={314--326},
  year={1975},
  publisher={Elsevier}
}

@article{W03,
Author = {Wagner, Allan R.},
ISSN = {02724995},
Journal = {Quarterly Journal of Experimental Psychology: Section B},
Number = {1},
Pages = {7},
Title = {Context-sensitive elemental theory.},
Volume = {56},
URL = {https://search.ebscohost.com/login.aspx?},
Year = {2003},
}

@article{G20,
Author = {George, David N.},
ISSN = {2329-8456},
Journal = {Journal of Experimental Psychology: Animal Learning and Cognition},
Number = {3},
Pages = {185 - 204},
Title = {The representation of stimulus conjunction in theories of associative learning: A context-dependent added-elements model.},
Volume = {46},
Year = {2020},
}

@article {G15,
    AUTHOR = {Ghirlanda, Stefano},
     TITLE = {On elemental and configural models of associative learning},
  JOURNAL = {Journal of Mathematical Psychology},
    VOLUME = {64/65},
      YEAR = {2015},
     PAGES = {8--16},
      ISSN = {0022-2496,1096-0880},
   MRCLASS = {91E40},
  MRNUMBER = {3322831},
       DOI = {10.1016/j.jmp.2014.11.003},
       URL = {https://doi.org/10.1016/j.jmp.2014.11.003},
}

@article {G18,
    AUTHOR = {Ghirlanda, Stefano},
     TITLE = {Studying associative learning without solving learning
              equations},
  JOURNAL = {Journal of Mathematical Psychology},
    VOLUME = {85},
      YEAR = {2018},
     PAGES = {55--61},
      ISSN = {0022-2496,1096-0880},
   MRCLASS = {91E40},
  MRNUMBER = {3852581},
       DOI = {10.1016/j.jmp.2018.07.003},
       URL = {https://doi.org/10.1016/j.jmp.2018.07.003},
}

@article{WM97,
  title={What's elementary about associative learning?},
  author={Wasserman, Edward A and Miller, Ralph R},
  journal={Annual review of psychology},
  volume={48},
  number={1},
  pages={573--607},
  year={1997},
  publisher={Annual Reviews 4139 El Camino Way, PO Box 10139, Palo Alto, CA 94303-0139, USA}
}

@article{A25,
  title={A mathematical formalization of the replaced elements model},
  author={Aguirre, Natham},
  journal={Behavior Research Methods},
  volume={57},
  number={11},
  pages={318},
  year={2025},
  doi = {10.3758/s13428-025-02858-1},
  publisher={Springer}
}

@article{ST11,
  title={Analysis of periodic nonautonomous inhomogeneous systems},
  author={Slane, Jean and Tragesser, Steven},
  journal={Nonlinear Dynamics and Systems Theory},
  volume={11},
  number={2},
  pages={183--198},
  year={2011}
}

@article{TL20,
  title={Inhibited elements model—implementation of an associative learning theory},
  author={Thorwart, Anna and Lachnit, Harald},
  journal={Journal of Mathematical Psychology},
  volume={94},
  pages={102310},
  year={2020},
  publisher={Elsevier}
}

@article{BVW00,
  title={A componential view of configural cues in generalization and discrimination in Pavlovian conditioning},
  author={Brandon, Susan E and Vogel, Edgar H and Wagner, Allan R},
  journal={Behavioural brain research},
  volume={110},
  number={1-2},
  pages={67--72},
  year={2000},
  publisher={Elsevier}
}

@article{RW72,
  title={A theory of Pavlovian conditioning: Variations in the effectiveness of reinforcement and non-reinforcement},
  author={Rescorla, Robert A and Wagner, A R},
  journal={Classical conditioning, Current research and theory},
  volume={2},
  pages={64--69},
  year={1972},
  publisher={Appleton-Century-Crofts}
}

@article{P94,
  title={Similarity and discrimination: a selective review and a connectionist model.},
  author={Pearce, John M},
  journal={Psychological Review},
  volume={101},
  number={4},
  pages={587},
  year={1994},
  publisher={American Psychological Association}
}

@article{ME02,
  title={Associative interference between cues and between outcomes presented together and presented apart: An integration},
  author={Miller, Ralph R and Escobar, Martha},
  journal={Behavioural Processes},
  volume={57},
  number={2-3},
  pages={163--185},
  year={2002},
  publisher={Elsevier}
}

\end{document}